\documentclass[12pt]{article}
\usepackage{amsmath}
\usepackage{amsfonts}
\usepackage{amsthm}
\usepackage{amssymb}
\usepackage{color}
\usepackage{enumerate}
\usepackage[mathscr]{euscript}

\newtheorem{theorem}{Theorem}[section]
\newtheorem{lemma}[theorem]{Lemma}

\newtheorem{example}[theorem]{Example}

\theoremstyle{definition}

\theoremstyle{remark}
\newtheorem{remark}[theorem]{Remark}

\newcommand{\tp}{{\rm TP}\hskip0.02cm}
\newcommand{\tc}{{\rm TC}\hskip0.02cm}

\def\f2{\mathbb{F}_2}

\def\lip{\hskip0.02cm{\rm Lip}\hskip0.01cm}
\def\supp{\hskip0.02cm{\rm supp}\hskip0.01cm}

\newcommand{\ep}{\varepsilon}

\newcommand{\1}{\mathbf{1}}

\newcommand\remove[1]{}

\begin{document}

\title{\LARGE On relations between transportation cost spaces and $\ell_1$}

\author{Sofiya Ostrovska and Mikhail~I.~Ostrovskii}

\date{\today}
\maketitle



\begin{abstract} The present paper deals with some structural properties of transportation cost spaces,
also known as Arens-Eells spaces, Lipschitz-free spaces and
Wasserstein spaces. The main results of this work are: (1) A
necessary and sufficient condition on an infinite metric space
$M$, under which the transportation cost space on $M$ contains an
isometric copy of $\ell_1$. The obtained condition is applied to
answer the open questions asked by C\'uth and Johanis (2017)
concerning several specific metric spaces. (2) The description of
the transportation cost space of a weighted finite graph $G$ as
the quotient $\ell_1(E(G))/Z(G)$, where $E(G)$ is the edge set and
$Z(G)$ is the cycle space of $G$.
\end{abstract}



\section{Introduction}

\subsection{Definitions}\label{S:Def}

Let $(M,d)$ be a metric space. Consider a real-valued finitely
supported function $f$ on $M$ with a zero sum, that is,
$\sum_{v\in \supp f}f(v)=0$. A natural and important
interpretation of such function, which goes back to at least
Kantorovich-Gavurin \cite{KG49}, is to consider it as a {\it
transportation problem}: one needs to transport certain product
from locations where $f(v)>0$ to locations where $f(v)<0$. More
formally, we represent $f$ as
$f=a_1(\1_{x_1}-\1_{y_1})+a_2(\1_{x_2}-\1_{y_2})+\dots+
a_n(\1_{x_n}-\1_{y_n})$, where $a_i\ge 0$, $x_i,y_i\in M$, and
$\1_u(x)$ for $u\in M$ is the {\it indicator function} of $\{u\}$.
Any such representation of $f$ will be called a {\it
transportation plan} for $f$. The {\it cost} of this
transportation plan (which consists in moving $a_i$ units from
$x_i$ to $y_i$) is defined as $\sum_{i=1}^n a_id(x_i,y_i)$. We
denote the real vector space of all transportation problems by
$\tp(M)$. We introduce the {\it transportation cost norm}
$\|f\|_{\tc}$ of a transportation problem $f$ as the minimal cost
over all such transportation plans. It is easy to see that the
minimum is attained - we consider finitely supported functions -
and that $\|\cdot\|_\tc$ is a norm. The completion of this normed
space is called a {\it transportation cost space} and is denoted
by $\tc(M)$.

Transportation cost spaces are of interest in many areas and are
studied under many different names (the most common ones are
included in the keywords section). We prefer to use the term {\it
transportation cost space} since it makes the subject of this work
instantly clear to a wide circle of readers and it also reflects
the historical approach leading to these notions. Interested
readers can find a review of the main definitions, notions, facts,
terminology and historical notes pertinent to the subject in
\cite{OO19+}.

In the theory of metric embeddings, transportation cost spaces are
of interest due to  the following observation by Arens and Eells
\cite{AE56}: The metric space $M$ admits a canonical isometric
embedding into $\tc(M)$, given by $v\mapsto \1_v-\1_O$, where $O$
is a base point in $M$.

For background information on the transportation cost spaces we
refer to \cite[Chapter 10]{Ost13} (where such spaces are called
{\it Lipschitz free spaces}) and \cite[Chapter 3]{Wea18} (where
such spaces are called {\it Arens-Eells spaces}).

\subsection{Motivation and statement of results}\label{S:Results}

In this work, new results pertinent to  the relations between the
structure of transportation cost spaces and $L_1$ are obtained.
Previously, such relations have been studied by many researchers,
see, for example,  \cite{APP19+}, \cite{Cha02}--\cite{DKO18+},
\cite{God10}, \cite{GL18}, \cite{IT03},
 \cite{KMO19+}, \cite{KN06}, \cite{NS07}, \cite{OO19+}.

From the historical perspective,  there are a few arguments in
favor of studying the Banach-space-theoretical structure of
transportation cost spaces. Below,  some of them are provided:

  (1) The linearization of the theory of cotype for metric spaces. This idea was put forth by
Bill Johnson. The idea is described in \cite[p.~223]{Bou86} and is
discussed in \cite{Nao18}.

 (2) The program of
using transportation cost spaces  to solve some important problems
of linear and nonlinear theory of Banach spaces suggested by
Godefroy-Kalton \cite{GK03}, who  used the name {\it
Lipschitz-free spaces}. This program was substantially advanced by
Kalton \cite{Kal04}.

(3) The observation by Arens and Eells \cite{AE56} (see above)
shows that transportation cost spaces are natural target spaces
for metric embeddings.  See \cite[Chapter 10]{Ost13}.\medskip

The main results of this paper include the following outcomes:

1) A necessary and sufficient condition for containment of
$\ell_1$ in $\tc(M)$ isometrically (Theorem \ref{T:Contell_1}).
This result relies on the previous studies in this direction,
namely, \cite{CJ17}, \cite[Theorem 3.1]{OO19+}, and \cite{KMO19+}.
It is used to answer the questions in \cite[Remark
10,p.~3416]{CJ17} which were left open in \cite{CJ17} and
\cite{OO19+}.

2) A  generalization of  the quotient over the cycle space
description of $\tc(G)$ for an unweighted finite graph $G$ (see
\cite[Proposition 10.10]{Ost13}) to the case of an arbitrary
finite metric space. It has to be mentioned that somewhat similar
descriptions for $\tc(\mathbb{R}^n)$ were obtained in \cite{CKK17}
and \cite{GL18}.

\section{Isometric copies of $\ell_1$ in $\tc(M)$ with infinite
$M$}

In what follows,  the standard terminology of matching theory
\cite{LP09} is used. We consider a metric space $M$ as an infinite
weighted complete graph, where the weight of each edge is defined
as the distance between its ends. Let $V$ be a subset of $M$ of
cardinality $2n$, $n\in\mathbb{N}$. If edges $\{x_iy_i\}_{i=1}^n$
with $x_i,y_i\in V$, $x_i\ne y_i$, do not have common ends, we
call $\{x_iy_i\}_{i=1}^n$ a {\it perfect matching} of the subgraph
of $M$ spanned by $V$. We define the {\it weight} of the perfect
matching  $\{x_iy_i\}_{i=1}^n$ as $\sum_{i=1}^n d(x_i, y_i)$.

\begin{theorem}\label{T:Contell_1} The space $\tc(M)$ contains $\ell_1$ isometrically
if and only if there exists a sequence of pairs
$\{x_i,y_i\}_{i=1}^\infty$ in $M$, with all elements distinct,
such that each set $\{x_iy_i\}_{i=1}^n$ of edges is a minimum
weight perfect matching in the subgraph spanned by
$\{x_i,y_i\}_{i=1}^n$.
\end{theorem}

\begin{proof} {\bf Sufficiency.} Let $\{x_i,y_i\}_{i=1}^\infty$ be
such a sequence. Set $f_i=(\1_{x_i}-\1_{y_i})/d(x_i,y_i)$. Since
each finite set $\{f_i\}_{i=1}^n$, $n\in\mathbb{N}$ is
isometrically equivalent to the unit vector basis of $\ell_1^n$ by
the argument of \cite{KMO19+}, we are done.\medskip

{\bf Necessity.} Recall that a metric space $M$ is called {\it
uniformly discrete} if there exists a constant $\delta
> 0$ such that
\[\forall u,v\in M~(u\ne v)\Rightarrow (d(u,v)\ge\delta).\]

To  prove the  necessity,  we shall consider the three cases:

\begin{enumerate}[(A)]

\item \label{C:A} The space $M$ has an accumulation point, which
means that there is a sequence $\{u_i\}_{i=1}^\infty$ of distinct
elements in $M$ and $u\in M$, such that
$\lim_{i\to\infty}d(u_i,u)=0$.

\item \label{C:B} The space $M$ is not uniformly discrete, but
does not have an accumulation point.

\item \label{C:C} The space $M$ is uniformly discrete.

\end{enumerate}

The proof of the necessity will be performed according to the
following steps:

\begin{itemize}

\item First, we derive that in Cases \eqref{C:A} and \eqref{C:B},
the space $M$ contains a sequence of pairs
$\{x_i,y_i\}_{i=1}^\infty$ such that for each $n\in\mathbb{N}$ the
set $\{x_iy_i\}_{i=1}^n$ of edges is a minimum weight perfect
matching in the subgraph spanned by $\{x_i,y_i\}_{i=1}^n$.

\item Further, it will be shown that in Case \eqref{C:C} either
$M$ contains a sequence of pairs $\{x_i,y_i\}_{i=1}^\infty$ such
that for each $n\in\mathbb{N}$ the set $\{x_iy_i\}_{i=1}^n$ of
edges is a minimum weight perfect matching in the subgraph spanned
by $\{x_i,y_i\}_{i=1}^n$, or $\tc(M)$ does not contain an
isometric copy of $\ell_1$.

\end{itemize}

Case \eqref{C:A}. If $M$ has an accumulation point $u$, and
$\{u_i\}_{i=1}^\infty$ is as in \eqref{C:A}, then either there are
infinitely many pairwise disjoint pairs $(i,j)$, $i,j\in \mathbb{N},~j>i$,
such that all of the triangle inequalities below are strict:

\begin{equation}\label{E:TriStrict} d(u_i,u_j)<d(u,u_i)+d(u,u_j),\end{equation} or, after eliminating
finitely many elements of the sequence, for all of the remaining
ones, the equality is reached in the triangle inequalities:
\begin{equation}\label{E:TriEq} d(u_i,u_j)=d(u,u_i)+d(u,u_j).\end{equation}

To finish the proof, it suffices to select two disjoint
subsequences $\{x_i\}_{i=1}^\infty$ and $\{y_i\}_{i=1}^\infty$  of
the sequence $\{u_i\}_{i=1}^\infty$ in such a way  that, for each
$n\in\mathbb{N}$, the set of edges $\{x_iy_i\}_{i=1}^n$ is the
minimum weight perfect matching in the complete graph with
vertices $\{x_i, y_i\}_{i=1}^n$.

First, this will be done in the easier  case where all triangles
inequalities are equalities, see \eqref{E:TriEq}. In this case,
one may let $x_i=u_{2i-1}$ and $y_i=u_{2i}$ and check that, by
virtue of equalities \eqref{E:TriEq}, any perfect matching in the
weighted graph with the vertices $\{u_i\}_{i=1}^{2n}$ is a minimum
weight matching since all of them have weight $\sum_{i=1}^{2n}
d(u_i,u)$.

In the case where \eqref{E:TriStrict} is satisfied for infinitely
many disjoint pairs $\{(i_t,j_t)\}_{t=1}^\infty$, we combine the
fact that the differences
$d(u,u_{i_t})+d(u,u_{j_t})-d(u_{i_t},u_{j_t})$ are strictly
positive and $\lim_{i\to\infty}d(u_i,u)=0$, and get that one can
pass to a subsequence (preserving the notation
$\{(i_t,j_t)\}_{t=1}^\infty$ for the subsequence) in such a way
that, for each $t$ and for any $k_1$ and $k_2$ which are $i_s$ or
$j_s$ for some $s>t$, the next inequality holds:
\begin{equation} \label{E:Excess}
d(u_{i_t},u_{k_1})+d(u_{j_t},u_{k_2})> \sum_{m=t}^\infty
d(u_{i_m},u_{j_m}).\end{equation}

Now, let \[x_1=u_{i_1}, y_1=u_{j_1},\dots,x_n=u_{i_n}, y_n=u_{j_n},
\dots.\]

To complete the proof, it remains to check that
$\{x_iy_i\}_{i=1}^n$ is the minimum weight perfect matching in the
complete graph spanned by $\{x_i,y_i\}_{i=1}^n$. It will be proved
inductively that $x_iy_i$, $i=1,2,\dots$, should be in the minimum
weight perfect matching. Assume that there is a minimum weight
perfect matching which does not contain $x_1y_1$. Then we have to
match $x_1=u_{i_1}$ with some $u_{k_1}$ and $y_1=u_{j_1}$ with
some $u_{k_2}$, where $k_1$ and $k_2$ are $i_s$ or $j_s$ for some
$s\in\{2,\dots,n\}$. But then \eqref{E:Excess} implies that the
sum $d(x_1,u_{k_1})+d(y_1,u_{k_2})$ is strictly larger than the
weight of the matching $\{x_iy_i\}_{i=1}^n$. Thus, $x_1y_1$ should
be in any minimum weight perfect matching.

It is clear that the same argument can be repeated for $x_2y_2$,
and so on. This completes the proof in the case where
\eqref{E:TriStrict} is satisfied for infinitely many disjoint
pairs, and, thus, for the case where $M$ has an accumulation
point.

\begin{remark} Our argument in Case \eqref{C:A} is close to the one in \cite[Theorem
5]{CJ17}. For the convenience of the reader, we presented our
argument in the form independent of \cite{CJ17}.
\end{remark}

Case \eqref{C:B}. Since the space $M$ is not uniformly discrete,
there are sequences $\{u_i\}_{i=1}^\infty$  and
$\{v_i\}_{i=1}^\infty$ in $M$ such that $u_i\ne v_i$ for all
$i\in\mathbb{N}$ and $\lim_{i\to\infty} d(u_i,v_i)=0$. The
following standard lemma will be applied.

\begin{lemma}\label{L:Subseq} Each sequence $\{w_i\}_{i=1}^\infty$ in a metric
space either contains a Cauchy subsequence or a $\delta$-separated
subsequence, where $\delta$ is some positive number and
\emph{$\delta$-separated} means that any two elements are at
distance at least $\delta$.
\end{lemma}

We apply Lemma \ref{L:Subseq} to the sequence
$\{u_i\}_{i=1}^\infty$ and keep the notation
$\{u_i\}_{i=1}^\infty$ for the obtained subsequence, and the
notation $\{v_i\}_{i=1}^\infty$ for the corresponding subsequence
of $\{v_i\}_{i=1}^\infty$.

If $\{u_i\}_{i=1}^\infty$ is a Cauchy sequence, we consider the
completion $\widetilde M$ of $M$ and a point $\widetilde
u\in\widetilde M$ such that $\lim_{n\to\infty}d(u_i,\widetilde
u)=0$, and construct, as in Case \eqref{C:A}, sequences
$\{x_i\}_{i=1}^\infty$ and $\{y_i\}_{i=1}^\infty$. Since these
sequences are subsequences of $\{u_i\}_{i=1}^\infty$, the
constructed subspace (see the proof of {\bf Sufficiency}) is not
only in $\tc(\widetilde M)$, but also in $\tc(M)$.
\medskip

Now, assume that  $\{u_i\}_{i=1}^\infty$ is $\delta$-separated for
some $\delta>0$. In this case, after omitting finite number of
terms in the sequences $\{u_i\}_{i=1}^\infty$  and
$\{v_i\}_{i=1}^\infty$, we may assume that $d(u_i,v_i)<\delta/4$
for every $i$. Denote the obtained subsequences of
$\{u_i\}_{i=1}^\infty$ and $\{v_i\}_{i=1}^\infty$ by
$\{x_i\}_{i=1}^\infty$  and $\{y_i\}_{i=1}^\infty$, respectively.

We have $d(x_i,y_i)<\delta/4$ and also, due to the triangle
inequality and $\delta$-separation of $\{x_i\}_{i=1}^\infty$, the
inequalities $d(x_i,y_j)>3\delta/4$ and $d(y_i,y_j)>\delta/2$ for
$i\ne j$ hold. These inequalities immediately imply that, for
every $n\in \mathbb{N}$ the set $\{x_iy_i\}_{i=1}^n$ of edges is a
minimum weight perfect matching in the subgraph spanned by
$\{x_i,y_i\}_{i=1}^n$. This completes our proof in Case
\eqref{C:B}.
\bigskip

Case \eqref{C:C}. We show that if we assume both

\begin{enumerate}[{\rm (i)}]

\item \label{I:i} That $\tc(M)$ contains a sequence
$\{f_i\}_{i=1}^\infty$ which is isometrically equivalent to the
unit vector basis of $\ell_1$,

and

\item\label{I:ii} That $M$ does not contain a sequence of pairs
$\{x_i,y_i\}_{i=1}^\infty$, such that each set
$\{x_iy_i\}_{i=1}^n$ of edges is a minimum weight perfect matching
in the subgraph spanned by $\{x_i,y_i\}_{i=1}^n$,

we get a contradiction.

\end{enumerate}

We start with a simple case where all $\{f_i\}$ are in $\tp(M)$.
Since our result is isometric, this will not complete the proof
for
 the general case. However,   in the easier
case $\{f_i\}\subset\tp(M)$, the main ideas are more transparent.
\medskip

For each element of the sequence $f_i$, we pick an optimal
transportation plan (it does not have to be unique):
\begin{equation}\label{E:TranspPlans}f_i=\sum_{j=1}^{m(i)}
a_{j,i}(\1_{x_{j,i}}-\1_{y_{j,i}}),~a_{j.i}>0.\end{equation} We
say that  $T_i:=\{x_{j,i},y_{j,i}\}_{j=1}^{m(i)}$ is the set of
{\it transportation pairs} for $f_i$.

\begin{lemma}\label{L:Alter} Each sequence $\{f_i\}_{i=1}^\infty\subset \tp(M)$ contains a subsequence $\{f_{i_n}\}_{n=1}^\infty$
satisfying at least one of the two conditions:

\begin{enumerate}[{\rm (1)}]

\item \label{I:Dis} In each $T_{i_n}$, there exists a
transportation pair such that the obtained set of transportation
pairs is pairwise disjoint.

\item \label{I:CommonPt} One can pick in each $T_{i_n}$ a
transportation pair such that, for the obtained set of
transportation pairs, there is an element $x\in M$ contained in
each of them.

\end{enumerate}
\end{lemma}

\begin{proof} This lemma can be proved by considering an
alternative: either there is an element $x\in M$ contained in
infinitely many transportation pairs or there is no such element.
\end{proof}

We apply Lemma \ref{L:Alter} to the sequence
$\{f_i\}_{i=1}^\infty$ equivalent to the unit vector basis of
$\ell_1$. Assume that condition \eqref{I:Dis} is satisfied for one
of its subsequences, which we still denote $\{f_i\}_{i=1}^\infty$.
Without loss of generality it may be assumed that the disjoint pairs are
$(x_{1,i},y_{1,i})$, denoted
 by $(x_i,y_i)$, for short. By our assumption \eqref{I:ii}, there
is $n\in \mathbb{N}$ such that $\{x_iy_i\}_{i=1}^n$ is not a
minimum weight perfect matching of the graph spanned by
$\{x_i,y_i\}_{i=1}^n$. Pick a minimum weight perfect matching for
this graph. Interchanging the labels in some pairs $(x_i,y_i)$ and
changing the signs of the corresponding $f_i$ if needed, one may
assume that the minimum weight perfect matching is of the form
$\{x_iy_{\pi(i)}\}_{i=1}^n$ for some bijection
$\pi:\{1,\dots,n\}\to \{1,\dots,n\}$.

Let $a_i>0$ be the quantity transported from $x_i$ to $y_i$ in the
plans \eqref{E:TranspPlans} (written in a different notation). Set
$a=\min_{1\le i\le n} a_i>0$.

Now,  consider the vector $\sum_{i=1}^n f_i$ and construct the
following transportation plan for it: the plan is close to being
the sum of plans \eqref{E:TranspPlans}, but
\begin{equation}\label{E:CanImprove}\sum_{i=1}^na_i(\1_{x_i}-\1_{y_i})\end{equation} in it is replaced by

\begin{equation}\label{E:Improved}\sum_{i=1}^n(a_i-a)(\1_{x_i}-\1_{y_i})+\sum_{i=1}^na(\1_{x_i}-\1_{y_\pi(i)}).\end{equation}
Since $\{x_iy_{\pi(i)}\}_{i=1}^n$ is a minimum weight perfect
matching while $\{x_iy_{i}\}_{i=1}^n$ is not, the obtained
transportation plan has a strictly smaller cost than the sum of
the plans \eqref{E:TranspPlans}. This leads to
\[\left\|\sum_{i=1}^n f_i\right\|_\tc<\sum_{i=1}^n\|f_i\|_\tc,\]
which is a contradiction with the hypothesis (i) that $\{f_i\}$
is isometrically equivalent to the unit vector basis of $\ell_1$.

Next, suppose that condition \eqref{I:CommonPt} of Lemma
\ref{L:Alter} is satisfied for a subsequence of
$\{f_i\}_{i=1}^\infty$, which we still denote
$\{f_i\}_{i=1}^\infty$. Relabelling and changing the signs of $f_i$ if needed, it may be assumed that
$\{x_{1,i},y_{1,i}\}_{i=1}^\infty$ are such that all $x_{1,i}$ are
the same, let us denote all of them by $x$. We claim that if
$\{f_i\}_{i=1}^\infty$ are isometrically equivalent to the unit
vector basis of $\ell_1$, then
\begin{equation}\label{E:Tree}\forall i,j\in \mathbb{N}\quad
d(x,y_{1,i})+d(x,y_{1,j})=d(y_{1,i},y_{1,j}).
\end{equation}
Assume the contrary, that is,
\begin{equation}\label{E:Strict}
\exists i,j\in \mathbb{N}\quad
d(x,y_{1,i})+d(x,y_{1,j})>d(y_{1,i},y_{1,j}).
\end{equation} Let $a=\min\{a_{1,i},a_{1,j}\}$. Consider the
function $f_i-f_j$, subtract the corresponding plans
\eqref{E:TranspPlans}, and make the following modification in the
resulting transportation plan. We replace the difference
\[a_{1,i}(\1_{x}-\1_{y_{1,i}})-a_{1,j}(\1_{x}-\1_{y_{1,j}})\]
by
\[(a_{1,i}-a)(\1_{x}-\1_{y_{1,i}})-(a_{1,j}-a)(\1_{x}-\1_{y_{1,j}})+a(\1_{y_{1,j}}-\1_{y_{1,i}}).
\]
The strict inequality \eqref{E:Strict} implies that it is a
strictly better plan. Thus,
$\|f_i-f_j\|_\tc<\|f_i\|_\tc+\|f_j\|_\tc$, and this contradiction
proves \eqref{E:Tree}.

Now, we introduce a new sequence $\{\widetilde x_i,\widetilde
y_i\}_{i=1}^\infty$ by setting $\widetilde x_i=y_{1, 2i-1}$ and
$\widetilde y_i=y_{1, 2i}$. As in the case (A)(2), it is easy to see that \eqref{E:Tree}
implies that the sequence of pairs $\{\widetilde x_i,\widetilde
y_i\}_{i=1}^\infty$ satisfies the condition in \eqref{I:ii}, and
we get a contradiction. This completes the proof in the case
$\{f_i\}\subset\tp(M)$.
\medskip

In the rest of this proof, our aim is to generalize the argument
we just presented to the case where elements
$\{f_i\}_{i=1}^\infty$ are not in $\tp(M)$ but in its completion
$\tc(M)$.
\medskip

By the standard description of the completion (see \cite[Section
3.11.4]{FHH+11}), each element $f\in\tc(M)$ can be presented as a
series of the form
\begin{equation}\label{E:SeriesRep}f=\sum_{k=1}^\infty\left(\sum_{i=s_k+1}^{s_{k+1}}a_i(\1_{x_i}-\1_{y_i})\right)\end{equation}
for some $0=s_1<s_2<\dots<s_k<\dots$, $\{x_i\}\subset M$,
$\{y_i\}\subset M$, and $\{a_i\}\subset \mathbb{R}^+$ with
\begin{equation}\label{E:Complet}\sum_{k=1}^\infty\left\|\sum_{i=s_k+1}^{s_{k+1}}a_i(\1_{x_i}-\1_{y_i})\right\|_\tc<\infty.\end{equation}
Furthermore, the norm $\|f\|_\tc$ is equal to the infimum of sums
of the form \eqref{E:Complet} over all representations
\eqref{E:SeriesRep}, in which we assume that the sums in brackets
are optimal transportation plans for the corresponding elements of
$\tp(M)$.

We are now considering Case \eqref{C:C} where the space $M$ is
uniformly discrete. Let
\begin{equation}\label{E:delta}\delta=\inf_{x\ne y\in
M}d(x,y).\end{equation} For $f\in \tp(M)$, we introduce
$\|f\|_1=\sum_{v\in\supp f}|f(v)|$ and extend this norm to
functions on $M$ with countable supports and absolutely summable
collections of values.

For $f\in\tp(M)$ the amount of the product which is to be
delivered is $\|f\|_1/2$ and the vector $f$ is in the kernel of
the linear functional on finitely supported vectors defined as the
sum of the values. By \eqref{E:delta}, we get that
$\delta\|f\|_1/2\le \|f\|_\tc$ for any $f\in\tp(M)$. For this
reason, the space $\tc(M)$ is continuously embedded in
$\ell_1(M)$, and is in the kernel of the functional defined as the
sum of all coordinates (this functional is naturally defined for
$f\in\ell_1(M)$).
\medskip

\begin{lemma}\label{L:L_1_below} Let $h\in\tc(M)$ and $g\in\tp(M)$  be such that the diameter of the support
of $g$ be  $\le D$. Then
$\|g+h\|_1\ge\frac2D(\|g\|_\tc-\|h\|_\tc)$.
\end{lemma}

Note that the conclusion of Lemma \ref{L:L_1_below} is nontrivial only if $\|h\|_\tc<\|g\|_\tc$.

\begin{proof}[Proof of Lemma \ref{L:L_1_below}] If $h\notin \tp(M)$, we can approximate $h$ arbitrarily well  - both in $\|\cdot\|_\tc$ and $\|\cdot\|_1$
 -  by vectors belonging
to $\tp(M)$. For this reason, we may assume that $h\in\tp(M)$.

Let us write an optimal transportation plan for $h$:

\begin{equation}\label{E:Rep_h} h=a_1(\1_{u_1}-\1_{v_1})+a_2(\1_{u_2}-\1_{v_2})+\dots+
a_n(\1_{u_n}-\1_{v_n}),\quad a_i>0.\end{equation}

By combining the corresponding terms, we may and shall assume that
none of the $v_i$ is equal to any of the $u_j$. We split the
transportation plan in \eqref{E:Rep_h} into two sums: (a) The sum
$h_1$ which contains the terms $a_i(\1_{u_i}-\1_{v_i})$ with at
most one of the elements $u_i,v_i$ being  in the support of $g$;
(b) The sum $h_2$ which contains those terms
$a_i(\1_{u_i}-\1_{v_i})$ for which both of the elements $u_i,v_i$
are in the support of $g$.

The important and easy-to-see observation is that
$\|g+h_1+h_2\|_1\ge \|g+h_2\|_1$. This is because each term of
$h_1$ can decrease the value of $g+h_2$ at some point, but it adds
the same amount elsewhere.

It remains to estimate $\|g+h_2\|_1$. Notice that
\[\|g+h_2\|_\tc\ge \|g\|_\tc-\|h_2\|_\tc\ge \|g\|_\tc-\|h\|_\tc,\]
and that - due to the choice of $h_2$ - the diameter of the
support of $g+h_2$ also does not exceed $D$. It remains to observe
that $\|g+h_2\|_\tc$ does not exceed the amount of product which
is to be moved - that is, $\|g+h_2\|_1/2$ - times the maximal
distance which this product has to travel - that is, $D$.
Therefore,
\[\|g+h_2\|_1\ge\frac2D\,\|g+h_2\|_\tc,\]
whence the conclusion follows.
\end{proof}

It has to be pointed out  that Lemma \ref{L:L_1_below} together
with equation \eqref{E:SeriesRep} along with  the fact that
$\|f\|_\tc$ is the infimum of sums \eqref{E:Complet} implies that
$m_i:=\|f_i\|_1>0$. Furthermore, it is easy to see that equation
\eqref{E:SeriesRep} and the fact that $\|f\|_\tc$ is the infimum
of sums \eqref{E:Complet} imply that, for each $i,m\in
\mathbb{N}$, one can write $f_i=S_i^m+R_i^m$ where
$S_i^m\in\tp(M)$, $R_i^m\in\tc(M)$, $1-2^{-m}\le \|S_i^m\|_\tc\le
1+2^{-m}$, $\|R_i^m\|_\tc\le 2^{-m+1}$, and $\|R_i^m\|_1\le
m_i/8$. Using the last inequality, one arrives at $\|S_i^m\|_1\ge
7m_i/8$.
\medskip

Writing $f_i^+$ and $f_i^-$ for the non-negative and non-positive
parts of $f_i$, we have $\|f_i^+\|_1=\|f_i^-\|_1=m_i/2$, and
therefore we can select
 in the support of $f_i^+$ a finite subset $V_i^+$ such that
$\sum_{v\in V_i^+}f_i(v)\ge \frac{7m_i}{16}$ and in the support of
$f_i^-$ a finite subset $V_i^-$ such that $\sum_{v\in
V_i^-}f_i(v)\le -\frac{7m_i}{16}$.

Let $i$ be fixed for a moment. As is well known (see
\cite[Proposition 3.16]{Wea18}), for each $m$ we can pick an
optimal transportation plan
\[S_i^m=\sum_{j=1}^{N_m}a_j^m(\1_{x_j^m}-\1_{y_j^m}),~ a_j^m>0,\] such that $S_i^m(x_j^m)>0$ and $S_i^m(y_j^m)<0$ for every $j=1,\dots,N_m$.
For each of $\{S_i^m\}_{m=1}^\infty$,   create a matrix whose
columns are labelled by elements of $V_i^+$ and whose rows are
labelled by elements of $V_i^-$. In the intersection of the column
corresponding to $x$ and the row corresponding to $y$, we record
the amount of product which is moved in such an optimal plan from
$x$ to $y$, while we put $0$ if nothing is moved. We claim that
the sum of all entries of the obtained matrix is at least
$\frac{4m_i}{16}=\frac{m_i}4$, i.e. $\sum_{j\in J_m} a^m_j\ge\frac
{m_i}4$ where $J_m = \{j\le N_m:~ (x^m_j, y^m_j)\in V^+_i\times
V^-_i\}$.

To prove this, let us introduce the following functions on $V_i^+$
and $V_i^-$, respectively:

\[P_i^m(v)=\begin{cases} S_i^m(v) &\hbox{ if } S_i^m(v)>0\hbox{ and }v\in V_i^+\\
0 &\hbox{ for all other }v\in V_i^+.
\end{cases}
\]

\[N_i^m(v)=\begin{cases} S_i^m(v) &\hbox{ if }S_i^m(v)<0\hbox{ and }v\in V_i^-\\
0 &\hbox{ for all other }v\in V_i^-.
\end{cases}
\]

It has to be  shown that in any optimal transportation plan for
$S_i^m$, a nontrivial part of product which is available at points
of $V_i^+$, that is $P_i^m$, should be transported to satisfy the
need at points of $V_i^-$, that is $N_i^m$. Evidently, each unit
of product available in the transportation problem $S_i^m$ at
points where $P_i^m>0$ should be moved to the points where
$S_i^m(v)<0$. These points can be the ones where $N_i^m(v)<0$, but
can be also some other points where $S_i^m(v)<0$. To estimate the
amount which should be moved to the points where $N_i^m(v)<0$ we
need some inequalities. Since $a^++b^+\ge (a+b)^+$, we have
\[\|P_i^m\|_1\ge \|f_i^+\|_1-\frac {m_i}{16}
-\|(R_i^m)^+\|_1\ge\frac{6m_i}{16}.\]

To see that some of these $\|P_i^m\|_1$ units have to go to the
points where $N_i^m<0$, we need to prove that outside $V_i^-$ the
total amount of negative values of $S_i^m$ is relatively small. In
fact, since $S_i^m=f_i-R_i^m$ such values can occur either at the
points $v$ where $f_i(v)<0$, but $v\notin V_i^-$, or because of
subtraction of $(R_i^m)^+$. Therefore, the total amount of this
need is $\le\frac{2m_i}{16}$. Thus at least
$\frac{4m_i}{16}=\frac{m_i}4$ of the product available at the
points where $P_i^m(v)>0$ should be transported to the points
where $N_i^m(v)<0$, as claimed.

Since the matrix described in the paragraph where we introduced
$V_i^+$ and $V_i^-$, is finite and a sum of its entries is at
least $\frac{m_i}{4}$, there exists  a subsequence of
$\{S_i^m\}_{m=1}^\infty$ such that, for some choice of $x_i\in
V_i^+$ and $y_i\in V_i^-$, the amount of transported product (with
respect to the picked above optimal transportation plan for
$S_i^m$) from $x_i$ to $y_i$ will be at least $\ep_i>0$, which is
a positive number depending only on $i$. Since subsequences
$\{S_i^m\}_{m=1}^\infty$ and $\{R_i^m\}_{m=1}^\infty$ also satisfy
the defining inequalities for $\{S_i^m\}_{m=1}^\infty$ and
$\{R_i^m\}_{m=1}^\infty$, we may assume without loss of
generality, that the subsequences are $\{S_i^m\}_{m=1}^\infty$ and
$\{R_i^m\}_{m=1}^\infty$ themselves.

From here on,  we follow the same line of argument as in the first
part of the proof (for Case \eqref{C:C}). Using the same reasoning
as in Lemma \ref{L:Alter}, one obtains, after taking subsequences
in $i$ the following alternatives: (1) the pairs $\{x_i,y_i\}$ are
disjoint; (2) they all have a common element (changing signs of
$f_i$ we can assume that all of $x_i$ are the same).

If alternative (1) holds, using the assumption \eqref{I:ii}, we
conclude that there exists a finite subcollection
$\{x_i,y_i\}_{i=1}^n$ such that in the subgraph spanned by it,
there is a perfect matching with a smaller weight. After that one
can apply the same ``improvement of a transportation plan'' as we
used when replacing \eqref{E:CanImprove} by \eqref{E:Improved}.
This improvement shows that there is $\tau_n>0$ such that
\begin{equation}\label{E:Impr_tau}\|S_1^m+\dots+S_n^m\|_\tc\le
\|S_1^m\|_\tc+\dots+\|S_n^m\|_\tc-\tau_n.\end{equation} A crucial
issue  is that this holds for every $m\in\mathbb{N}$. More
precisely, formula \eqref{E:Improved} shows that $\tau_n$ can be
chosen to be the product of $\min_{1\le i\le n}\ep_i$ and the
difference between the weights of the matching
$\{x_iy_i\}_{i=1}^n$ and the minimum weight perfect matching in
the subgraph spanned by $\{x_i,y_i\}_{i=1}^n$.\medskip

Therefore, one obtains
\[\begin{split}\|f_1+\dots+f_n\|_\tc&\le
\|S_1^m+\dots+S_n^m\|_\tc+\|R_1^m+\dots+R_n^m\|_\tc\\&\le
n\cdot(1+2^{-m})-\tau_n+n2^{-m+1}.\end{split}\]

Since this inequality holds for every $m$ and $\tau_n$ does not
depend on $m$, we can pick $m$ in such a way that the number in
the rightmost side of the last inequality is $<n$. This gives a
contradiction with the assumption that $\{f_i\}$ is isometrically
equivalent to the unit vector basis of $\ell_1$.

Finally, consider alternative (2): the pairs $\{x_i,y_i\}$ have a
common point. As above, one may assume that this common point
coincides with all of $x_i$ and denote it by $x$. Similarly to the
argument above, the goal is to prove the equalities in some of the
triangle inequalities. Now the desired equalities are:
\begin{equation}\label{E:Tree2}
d(x,y_i)+d(x,y_j)=d(y_i,y_j) \quad \hbox{ for } i\ne j.
\end{equation}

The proof goes according to the same steps as above. If one of the
triangle inequalities is strict, that is
$d(x,y_i)+d(x,y_j)>d(x_i,x_j)$, we can find $\tau_{i,j}>0$ such
that $\|S_i^m-S_j^m\|_\tc\le
\|S_i^m\|_\tc+\|S_j^m\|_\tc-\tau_{i,j}$ for every $m$. From here,
we derive $\|f_i-f_j\|_\tc<\|f_i\|_\tc+\|f_j\|_\tc$, contrary to
the assumption that $\{f_i\}$ is isometrically equivalent to the
unit vector basis of $\ell_1$.

After establishing  \eqref{E:Tree2}, we complete the proof as in
the previous case.
\end{proof}

\begin{example} As an application of Theorem \ref{T:Contell_1}, we use it to answer the questions
on isometric presence of $\ell_1$ in $\tc(M)$ for metric spaces
$M$ listed in \cite[Remark 10,p.~3416]{CJ17}, for which the answer
has not been known.
 In all of the examples $M=\{v_n\}_{n=1}^\infty$. The
metrics on $M$ are defined for $n>k$ as follows:
\begin{enumerate}[{\rm (a)}]\setlength\itemsep{0.3em}
\item\label{I:a} $\rho(v_k,v_n)=k+n-\frac1k$ \item \label{I:b}
$\rho(v_k,v_n)=2-\frac1k+\frac1n$ \item \label{I:c}
$\rho(v_k,v_n)=2-\frac1k-\frac1{2n}$ \item \label{I:d}
$\rho(v_k,v_n)=1+\frac1n$ \item \label{I:e}
$\rho(v_k,v_n)=1+\frac1{2k}+\frac1n$
\end{enumerate}
\end{example}

Using Theorem \ref{T:Contell_1} we can prove that in all of the
examples the answer is negative - the corresponding transportation
cost spaces do not contain isometric copies of $\ell_1$.

In each of the cases, we prove that, for any selected sequence
$\{x_i,y_i\}_{i=1}^\infty$ of pairs of distinct elements in the
metric space, one can find $m$ such that the set
$\{x_iy_i\}_{i=1}^m$ of edges is not a minimum weight perfect
matching in the complete graph spanned by $\{x_i,y_i\}_{i=1}^m$
(with weight of each edge equal to the distance between its ends).
\medskip

The main observation here is that no matter how the sequence
$\{x_i,y_i\}_{i=1}^\infty$ is selected, it is possible to pick two
pairs $(x_j,y_j)$ and $(x_m,y_m)$, $j<m$, such that the indices of
vertices $x_m$ and $y_m$ in the sequence $\{v_n\}_{n=1}^\infty$
are larger than the indices of $x_j$ and $y_j$. Without loss of
generality we may assume that indices of $x_j,y_j,x_m,y_m$ are
$q_1<q_2<q_3<q_4$, respectively.

This will immediately imply the desired conclusion of the previous
paragraph as soon as it will be derived  that
$d(x_j,x_m)+d(y_j,y_m)<d(x_j,y_j)+d(x_m,y_m)$. Hence, it remains
to verify this inequality in cases \eqref{I:a}-\eqref{I:e}.
Indeed, direct calculations lead to the following inequalities,
which are obviously true:
\medskip

\eqref{I:a}
$\displaystyle{q_1+q_3-\frac1{q_1}+q_2+q_4-\frac1{q_2}<q_1+q_2-\frac1{q_1}+q_3+q_4-\frac1{q_3}}$
as $\displaystyle{\frac1{q_3}<\frac1{q_2}}$.

\eqref{I:b}
$\displaystyle{2-\frac1{q_1}+\frac1{q_3}+2-\frac1{q_2}+\frac1{q_4}<2-\frac1{q_1}+\frac1{q_2}+2-\frac1{q_3}+\frac1{q_4}}$
as $\displaystyle{\frac2{q_3}<\frac2{q_2}}$.

\eqref{I:c}
$\displaystyle{2-\frac1{q_1}-\frac1{2q_3}+2-\frac1{q_2}-\frac1{2q_4}<2-\frac1{q_1}-\frac1{2q_2}+2-\frac1{q_3}-\frac1{2q_4}}$
as $\displaystyle{\frac1{2q_3}<\frac1{2q_2}}$.

\eqref{I:d}
$\displaystyle{1+\frac1{q_3}+1+\frac1{q_4}<1+\frac1{q_2}+1+\frac1{q_4}}$
as $\displaystyle{\frac1{q_3}<\frac1{q_2}}$.

\eqref{I:e}
$\displaystyle{1+\frac1{2q_1}+\frac1{q_3}+1+\frac1{2q_2}+\frac1{q_4}<1+\frac1{2q_1}+\frac1{q_2}+1+\frac1{2q_3}+\frac1{q_4}}$
as $\displaystyle{\frac1{2q_3}<\frac1{2q_2}}$.

\section{Canonical description of $\tc(M)$ as a quotient of $\ell_1$ for
finite metric space $M$}

The goal of this section is a generalization for an arbitrary
finite metric space of the known description \cite[Proposition
10.10]{Ost13}  of transportation cost spaces for finite unweighted
graphs as quotients of finite-dimensional $\ell_1$.
\medskip

Let $M$ be a finite metric space with $n$ elements. It can be
viewed as a weighted complete graph $K_n$, where the weight of the
edge joining $u$ and $v$ is the distance $d(u,v)$. The weighted
$\ell_1$-space on the edge set $E(K_n)$ will be introduced as
follows. Given $f:E(K_n)\to \mathbb{R}$, denote by $f_{uv}$ the
value of this function on the edge $uv$. The norm of $f$ is
defined as:
 \[\|f\|_{1,d}:=\sum_{uv\in
E(K_n)} |f_{uv}|d(u,v).\] The normed space obtained in this way
will be denoted by $\ell_{1,d}=\ell_{1,d}(E(K_n))$. It can be
readily seen that it is an $\frac{n(n-1)}2$-dimensional space
isometric to $\ell_1^{n(n-1)/2}$.

Further, let us fix an  orientation on the edges of $K_n$. Notice
that  only intermediate objects and results rather than final
outcomes  will depend on it. For this reason, it is customary to
say that we select a {\it reference orientation}. Consider  a
cycle $C$ in $K_n$ and pick one of the two possible orientations
of $C$ satisfying the following condition: each vertex of $C$ is a
head of exactly one edge and a tail of exactly one edge. Having
done so,  we introduce the {\it signed indicator function}
$\chi_C\in \ell_{1,d}$ of the cycle $C$ by
\begin{equation}\label{E:SignInd}
\chi_C(e)=\begin{cases} 1 & \hbox{ if }e\in C\hbox{ and its
orientations in $C$ and $G$ are the same}\\
-1 & \hbox{ if }e\in C\hbox{ but its orientations in $C$ and $G$
are different}
\\
0 & \hbox{ if }e\notin C,
\end{cases}\end{equation}
where $e$ is used to denote edges in $K_n$.

The span of this set of functions in $\ell_{1,d}$ is denoted by
$Z$ and called the {\it cycle space} (or the {\it flow space} in
some sources). The following assertion holds.
\begin{theorem}\label{T:GenCyc}  $\tc(M)=\ell_{1,d}/Z$.
\end{theorem}

\begin{proof} Since the spaces $\ell_{1,d}/Z$ and
$\tc(M)$ are finite-dimensional, it suffices to show that the dual
space $(\ell_{1,d}/Z)^*$ can be in a natural way identified with
the space $\lip_0(M)$, which is known to coincide with
$(\tc(M))^*$, see \cite[Theorem 10.2]{Ost13}.
\medskip

To begin with, let us introduce the spaces $\ell_{\infty,d}$ and
$\ell_{2,d}$ as spaces of real-valued functions on $E(K_n)$ with
the norms
\[\|f\|_{\infty,d}=\max_{uv\in E(K_n)}\frac{|f_{uv}|}{d(u,v)}\]
and \begin{equation}\label{E:L2Norm}\|f\|_{2,d}=\left(\sum_{uv\in
E(K_n)}|f_{uv}|^2\right)^{\frac12},\end{equation} respectively.

It is clear that  $\ell_{2,d}$ is an inner product space, in which
 the notion of orthogonality is naturally defined. We denote the inner product inducing the
norm \eqref{E:L2Norm} by $\langle\cdot,\cdot\rangle$.

The subspace of $\ell_{2,d}$ orthogonal to the cycle space $Z$ is
denoted by $B$ and is called the {\it cut space} or {\it cut
subspace}. Observe that by virtue of  \eqref{E:L2Norm},   $B$ does
not really depend on the distance $d$, but only on the size of
$M$. One has a direct, orthogonal in $\ell_{2,d}$, decomposition
\begin{equation}\label{E:ZBDecomp}\ell_{2,d}=Z\oplus B.\end{equation}

Next, we apply the standard duality result, which, generally
speaking, states that that the dual of the quotient space
$\mathcal{X}/\mathcal{Y}$ is isometric to the subspace
$\mathcal{Y}^\bot:=\{f\in \mathcal{X}^*:~\forall y\in
\mathcal{Y},~f(y)=0\}$. Observing that our choice of norms on
$\ell_{1,d}$  and $\ell_{\infty,d}$ is such that
$\ell_{\infty,d}=(\ell_{1,d})^*$ with the pairing given by
$g(f)=\langle g, f\rangle$, one concludes  that the dual space of
the quotient space $\ell_{1,d}/Z$ is naturally isometric to the
space $B_\infty$, where  $B_\infty$ stands for the space $B$
endowed with its $\ell_{\infty,d}$-norm.
\medskip

To complete our argument it is convenient to use another
description of $B$. Denote by $\ell_2(M)$ the space $\mathbb{R}^M$
with its Euclidean norm. Let $D$ be defined as a matrix whose rows
are labelled using elements of $M$, whose columns are labelled
using (oriented) edges of $K_n$ and the $ve$-entry is given by
\[
d_{ve}=\begin{cases} 1, & \hbox{ if }v\hbox{ is the head of }e,\\
-1, & \hbox{ if }v\hbox{ is the tail of }e,\\
0, & \hbox{ if $v$ is not incident to }e.\end{cases} \]

The description of $B$ which we are going to use is that that $B$
is the image of $\ell_2(M)$ under the action of $D^T$ with $D^T$
being the transpose of the matrix $D$. See \cite[p.~315]{Ost13}
noting that the result described there holds, in particular, for
the cut space of the complete graph.

Therefore, each $b\in B$ can be represented as $b=D^Tf$, implying
that $b(uv)=h(u)-h(v)$ for some $h:M\to \mathbb{R}$ and all
oriented edges $uv$, where $u$ is the head and $v$ is the tail. It
is clear that addition of a constant to the function $h$ does not
change $D^Th$, so one may assume $h(O)=0$, that is,
$h\in\lip_0(M)$. Clearly, the Lipschitz constant of $h$ is
equal to
\[\lip(h)=\max_{uv\in E(K_n)}\frac{|h(u)-h(v)|}{d(u,v)}=\|b\|_{\infty,d}.\]
Thus, we have established a natural isometry between $B_\infty$
and $\lip_0(M)$.\end{proof}

\section*{Acknowledgement}

The second author gratefully acknowledges the support by the
National Science Foundation grant NSF DMS-1700176 and by St.
John's University. We would like to thank the referee for the
careful reading of the paper and numerous corrections.


\begin{small}

\renewcommand{\refname}{
\end{small}

\textsc{Department of Mathematics, Atilim University, 06830
Incek,\\ Ankara, TURKEY} \par \textit{E-mail address}:
\texttt{sofia.ostrovska@atilim.edu.tr}\par\medskip

\textsc{Department of Mathematics and Computer Science, St. John's
University, 8000 Utopia Parkway, Queens, NY 11439, USA} \par
  \textit{E-mail address}: \texttt{ostrovsm@stjohns.edu} \par


\section*{References}}

\begin{thebibliography}{99}

\bibitem{APP19+} R.\,J.~Aliaga, C.~Petitjean,
A.~Proch\'azka, Embeddings of Lipschitz-free spaces into $\ell_1$,
{\tt arXiv:1909.05285}.

\bibitem{AE56} R.\,F.~Arens, J.~Eells, Jr., On embedding uniform and topological spaces, {\it Pacific J. Math.},
{\bf  6}  (1956), 397--403.

\bibitem{Bou86} J.~Bourgain, The metrical interpretation
of superreflexivity in Banach spaces, {\it Israel J. Math.}, {\bf
56} (1986), no. 2, 222--230.

\bibitem{Cha02} M.\,S.~Charikar, Similarity estimation techniques from
rounding algorithms. {\it Proceedings of the Thirty-Fourth Annual
ACM Symposium on Theory of Computing}, 380--388, ACM, New York,
2002.

\bibitem{CD16} M.~C\'uth, M.~Doucha, Lipschitz-free spaces over ultrametric spaces. {\it Me\-di\-terr. J. Math.} {\bf 13} (2016), no. 4,
1893--1906.

\bibitem{CDW16} M.~C\'uth, M.~Doucha, P.~Wojtaszczyk, On the
structure of Lipschitz-free spaces. {\it Proc. Amer. Math. Soc.}
{\bf 144} (2016), no. 9, 3833--3846.

\bibitem{CJ17} M.~C\'uth, M.~Johanis,
Isometric embedding of $\ell_1$ into Lipschitz-free spaces and
$\ell_\infty$ into their duals. {\it Proc. Amer. Math. Soc.} {\bf
145} (2017), no. 8, 3409--3421.

\bibitem{CKK17} M.~C\'uth, O.\,F.\,K.~Kalenda, P.~Kaplick\'y,
Isometric representation of Lipschitz-free spaces over convex
domains in finite-dimensional spaces. {\it Mathematika} {\bf 63}
(2017), no. 2, 538--552.

\bibitem{DKP16} A.~Dalet, P.L.~Kaufmann, A.~Proch\'azka,
Characterization of metric spaces whose free space is isometric to
$\ell_1$. {\it Bull. Belg. Math. Soc. Simon Stevin} {\bf 23}
(2016), no. 3, 391--400.

\bibitem{DKO18+} S.\,J. Dilworth, D.~Kutzarova,
M.\,I.~Ostrovskii, Lipschitz-free spaces on finite metric spaces,
{\it Canad. J. Math.}, 72 (2020), 774--804, {\tt
DOI:10.4153/S0008414X19000087}.

\bibitem{FHH+11} M.~Fabian, P.~Habala, P.~H\'ajek, V.~Montesinos,
V.~Zizler, {\it Banach space theory. The basis for linear and
nonlinear analysis}. CMS Books in Mathematics/Ouvrages de
Math\'ematiques de la SMC. Springer, New York, 2011.

\bibitem{Fre10} M.~Fr\'echet, Les dimensions d'un ensemble abstrait, {\it Math. Ann.},
{\bf  68}  (1910),  no. 2, 145--168.

\bibitem{God10}
A.~Godard,  Tree metrics and their Lipschitz-free spaces, {\it
Proc. Amer. Math. Soc.}, {\bf 138} (2010), no. 12, 4311--4320.

\bibitem{GK03} G.~Godefroy, N.\,J.~Kalton, Lipschitz-free Banach spaces, {\it Studia Math.},
{\bf 159}  (2003),  no. 1, 121--141.

\bibitem{GL18} G.~Godefroy, N.~Lerner, Some natural subspaces and quotient spaces of $L_1$.
{\it Adv. Oper. Theory} {\bf 3} (2018), no. 1, 61--74.

\bibitem{IT03} P.~Indyk, N.~Thaper, Fast image retrieval via embeddings,
in: {\it ICCV 03: Proceedings of the 3rd International Workshop on
Statistical and Computational Theories of Vision}, 2003.

\bibitem{Kal04} N.\,J.~Kalton,
Spaces of Lipschitz and H\"older functions and their applications,
{\it Collect. Math.}, {\bf 55} (2004), 171--217.

\bibitem{KG49} L.\,V.~Kantorovich, M.\,K.~Gavurin, Application of mathematical methods in the analysis of cargo flows (Russian),
in: {\it Problems of improving of transport efficiency}, USSR
Academy of Sciences Publishers, Moscow, 1949, pp. 110--138.

\bibitem{KMO19+} S.\,S.~Khan, M.~Mim, M.\,I.~Ostrovskii, Isometric copies of $\ell_\infty^n$ and $\ell_1^n$ in transportation cost
spaces on finite metric spaces, in: {\it The Mathematical Legacy
of Victor Lomonosov. Operator Theory}, pp.~189--203, De Gruyter,
2020; {\tt DOI:10.1515/9783110656756-014}.


\bibitem{KN06} S. Khot, A.~Naor, Nonembeddability
theorems via Fourier analysis, {\it Math. Ann.}, {\bf 334} (2006),
821--852.

\bibitem{LP09} L.~Lov\'asz, M.\,D.~Plummer, {\it Matching theory.}
Corrected reprint of the 1986 original. AMS Chelsea Publishing,
Providence, RI, 2009.

\bibitem{Nao18} A.~Naor, Metric dimension reduction: a snapshot of the Ribe program, {\it Proc. Int. Cong. of
Math. - 2018}, Rio de Janeiro, Vol. {\bf 1}, 759--838.

\bibitem{NS07} A. Naor, G. Schechtman, Planar Earthmover is not in
$L_1$, {\it SIAM J. Computing}, {\bf 37} (2007), 804--826.

\bibitem{OO19+} S.~Ostrovska, M.\,I.~Ostrovskii, Generalized transportation cost spaces, {\it Me\-di\-terr. J. Math.},
{\bf 16} (2019), no. 6, Paper No. 157.

\bibitem{Ost13} M.\,I.~Ostrovskii, {\it Metric Embeddings: Bilipschitz and Coarse Embeddings into Banach Spaces},
de Gruyter Studies in Mathematics, {\bf 49}. Walter de Gruyter \&\
Co., Berlin, 2013.

\bibitem{Wea18} N.~Weaver, {\it Lipschitz algebras}, Second edition, World Scientific
Publishing Co. Pte. Ltd., Hackensack, NJ, 2018.


\end{thebibliography}
\end{document}